\documentclass{amsart}
\usepackage{amsmath,amssymb,amscd,latexsym,verbatim}

\usepackage{amsmath,amssymb,amscd,latexsym}
\usepackage{enumerate}
\usepackage[all]{xy}
\xyoption{curve}
\SelectTips{eu}{}
\CompileMatrices
 
\newcommand{\HS}{\operatorname{HS}}

\newcommand{\GL}{\operatorname{GL}}

\newcommand{\Ker}{\operatorname{Ker}}

\newcommand{\ini}{\operatorname{in}}
\newcommand{\gin}{\operatorname{gin}}
\newcommand{\Min}{\operatorname{Min}}
\newcommand{\Ass}{\operatorname{Ass}}
\newcommand{\length}{\operatorname{length}}

\newcommand{\NN}{{\mathbb N}}
\newcommand{\QQ}{{\mathbb Q}}

\newcommand{\ZZ}{{\mathbb Z}}

\newcommand{\KK}{{\mathcal K}}
\newcommand{\CC}{{\mathcal C}}
\newcommand{\GG}{{\mathcal G}}

\theoremstyle{plain}

\newtheorem{theorem}{Theorem}[section]
\newtheorem{proposition}[theorem]{Proposition}
\newtheorem{lemma}[theorem]{Lemma}

\newtheorem{corollary}[theorem]{Corollary}

\theoremstyle{definition}

\newtheorem{definition}[theorem]{Definition}

\newtheorem{remark}[theorem]{Remark}

\newtheorem{chunk}[theorem]{}

\theoremstyle{remark}

\numberwithin{equation}{theorem}

\begin{document}

\title[Universal Gr\"obner bases for maximal minors]
{Universal Gr\"obner bases for maximal minors}
\author[A.~Conca]{Aldo~Conca}
\address{Dipartimento di Matematica, 
Universit\`a di Genova, Via Dodecaneso 35, 
I-16146 Genova, Italy}
\email{conca@dima.unige.it}
\author[E.~De Negri]{Emanuela De Negri}
\address{Dipartimento di Matematica, 
Universit\`a di Genova, Via Dodecaneso 35, 
I-16146 Genova, Italy}
\email{denegri@dima.unige.it}
\author[E. Gorla]{Elisa Gorla}
\address{Institut de Math\'ematiques, 
Universit\'e de Neuch\^atel, Rue Emile-Argand 11, 
CH-2000 Neuch\^atel, Switzerland} 
\email{elisa.gorla@unine.ch}
\date{\today}
\thanks{The first two authors were partially supported by the Italian
  Ministry of Education, University and Research through the PRIN
  2010-11  ``Geometria delle Variet\`a Algebriche". The third author
  was partially supported by the Swiss National Science Foundation
  under grant no. PP00P2\_123393.} 



%\date{\today}
 
%------------------------------------------------------------------------
%
%
%
%------------------------------------------------------------------------

\maketitle

\begin{abstract}
Bernstein, Sturmfels and Zelevinsky proved in 1993 that the maximal
minors of a matrix of variables form a universal Gr\"obner basis.  
We present a very short proof of this result, along with broad
generalization to matrices with multi homogeneous structures.  
Our main tool is a rigidity statement for radical  Borel  fixed ideals
in multigraded polynomial rings.  
\end{abstract}

\section*{Introduction} 

A set $G$  of polynomials  in a polynomial ring $S$ over a field is
said to be  a universal Gr\"obner basis if $G$  is a Gr\"obner basis
with respect to every term order on $S$. Twenty years ago Bernstein,
Sturmfels and Zelevinsky proved in  \cite{BZ,SZ} that the set of the
maximal minors of an $m\times n$ matrix  of variables $X$ is a
universal Gr\"obner basis. Indeed, in \cite{SZ} the assertion is
proved for certain values of $m,n$ and the general problem is reduced
to a combinatorial statement that it is then proved in \cite{BZ}.
Kalinin  gave in \cite{K} a different proof of this result.  Boocher
proved in \cite{B} that any initial ideal  of the ideal $I_m(X)$ of
maximal minors of $X$ has a linear resolution (or, equivalently in
this case, defines a Cohen-Macaulay ring).   

The goal of this paper is twofold. First, we give a quick proof of the
results mentioned above. Our proof is based on a specialization
argument, see Section \ref{origUGB}. Second, we show that similar
statements hold in a more general setting, for matrices of linear
forms satisfying certain homogeneity conditions. More precisely, in
Section \ref{colgrad} we show that the set of the maximal minors of an
$m\times n$ matrix $L=(L_{ij})$ of linear forms is a universal
Gr\"obner basis, provided that $L$ is column-graded. By this we mean
that the entries $L_{ij}$ belong to a polynomial ring with a standard
$\ZZ^n$-graded structure, and that $\deg L_{ij}=e_j\in \ZZ^n$. Under
the same assumption we show that every initial ideal of $I_m(L)$ has a
linear resolution. Furthermore the projective dimension of $I_m(L)$
and of its initial ideals is $n-m$, unless $I_m(L)=0$ or a column of
$L$ is identically $0$ (notice that, under these assumptions, the
codimension of $I_m(L)$ can be smaller than $n-m+1$).    

If instead $L$ is row-graded, i.e. $\deg L_{ij}=e_i\in \ZZ^m$, then we
prove in Section \ref{rowgrad} that $I_m(L)$ has a universal Gr\"obner
basis of elements of degree $m$ and that every initial ideal of
$I_m(L)$ has a linear resolution, provided that $I_m(L)$ has the
expected codimension. Notice that in the row-graded case the maximal
minors do not form a universal Gr\"obner basis in general (since every
maximal minor might have the same initial term).   

The proofs of the statements in Sections \ref{colgrad} and
\ref{rowgrad} are based on a rigidity property of radical Borel fixed
ideals in a multigraded setting. This property has been observed by
Cartwright and Sturmfels  in \cite{CS} and  by Aholt, Thomas,  and
Sturmfels in \cite{ATS}, in special cases.  In a polynomial ring with a
standard $\ZZ^m$-grading, one can take generic initial ideals with
respect to the the product of general linear groups preserving the
grading. Such generic initial ideals are Borel fixed. The main theorem
of Section \ref{radBor} asserts that if two Borel fixed ideals $I,J$
have the same Hilbert series and $I$ is radical, then $I=J$. This is
the rigidity property that we referred to, and which has very strong
consequences. For instance if $I$ is  Cohen-Macaulay, radical and
Borel fixed, then all the multihomogeneous ideals with the same
multigraded Hilbert series are Cohen-Macaulay and radical as well.

Extensive computations performed with CoCoA \cite{Cocoa}  led to the
discovery of the results and examples presented in this paper. 
We thank Christian Krattenthaler for suggesting the elegant proof of
formula \eqref{rg8}. 
This work was done while the authors were at MSRI for the 2012-13
special year in commutative algebra. We thank the organizers and the
MSRI staff members for the invitation and for the warm hospitality.

\section{A simple proof of the universal GB theorem}\label{origUGB}

Let $K$ be a field, $S=K[x_{ij} : 1\leq i \leq m,  \ 1\leq j \leq
n]$. Let $X=(x_{ij})$ be an $m\times n$ matrix of indeterminates, and
let $I_m(X)$ be the ideal generated by the maximal minors of $X$. The
goal of this section is giving a quick proof of the following result
of  Bernstein, Sturmfels, Zelevinsky \cite{BZ,SZ}, and  Boocher
\cite{B}:  

\begin{theorem}\label{BSZB}   The set of maximal minors of $X$ is a
  universal Gr\"obner basis of $I_m(X)$, i.e., a Gr\"obner basis of
  $I_m(X)$ with respect to all the term orders. Furthermore every
  initial ideal of $I_m(X)$  has the same Betti numbers as $I_m(X)$.  
\end{theorem} 

We need the following ``Hilfssatz":

\begin{lemma}\label{help}  
Let $R$ be a standard graded $K$-algebra,  let $M,N,T$ be finitely
generated graded modules $R$-modules, and $J=(y_1,\dots,y_s) \subset
R$ be a homogeneous ideal. Suppose that:  
\begin{itemize}
\item[(1)]  there exists  a surjective graded $R$-homorphism $f:T\to N$. 
\item[(2)] $M$ and $N$ have the same Hilbert series, 
\item[(3)] $M/JM$ and $T/JT$ have the same Hilbert series, 
\item[(4)] $y_1,\dots,y_s$ is $M$-regular sequence. 
\end{itemize} 
Then $f$ is an isomorphism  and $y_1,\dots,y_s$ is a $T$-regular sequence. 
\end{lemma} 
\begin{proof}  
We denote by $\HS(M,x)\in \QQ[|x|][x^{-1}]$ the Hilbert series of a
finitely generated graded $R$-module $M$.   
For $i=1,\dots,s$ set $J_i=(y_1,\dots,y_i)$,  $T_i=T/J_iT$,
$d_i=\deg(y_i)$ and $g_i(x)=\prod _{j=1}^i (1-x^{d_j})\in \QQ[x]$.
Furthermore set $T_0=T$ and for $i=0,\dots, s-1$ denote by  $K_{i+1}$
the submodule $\{ m\in T_i : y_{i+1}m=0\}$ of $T_i$ shifted by
$-d_{i+1}$. Finally set $K_0=\Ker f$.  For $i\geq 0$ we have an exact
complex:  
$$0\to K_{i+1} \to T_i(-d_{i+1})\to T_i \to T_{i+1} \to 0$$
where the middle map is multiplication by $y_{i+1}$. Additivity of
dimensions on exact sequences of vector spaces yields:  
$$\HS(T_{i+1},x)=(1-x^{d_{i+1}})\HS(T_i,x)+\HS(K_{i+1},x)$$
and hence 
$$\HS(T_{i+1},x)=g_{i+1}(x)\HS(T,x)+\sum_{j=1}^{i+1} g_{i+1-j}(x) \HS(K_j,x).$$
Since 
$$\HS(T,x)=\HS(N,x)+\HS(K_0,x)$$  we may write
$$\HS(T_{i+1},x)=g_{i+1}(x)\HS(N,x)+\sum_{j=0}^{i+1} g_{i+1-j}(x) \HS(K_j,x)$$
and in particular, for $i+1=s$, 
$$\HS(T/JT,x)=g_s(x)\HS(N,x)+\sum_{j=0}^s g_{s-j}(x) \HS(K_j,x).$$
Since $\HS(N,x)=\HS(M,x)$ and $y_1,\dots,y_s$ is an $M$-regular sequence we obtain 
$$\HS(T/JT,x)=\HS(M/JM,x)+\sum_{j=0}^s g_{s-j}(x) \HS(K_j,x).$$
Hence, by assumption (3), we have 
$$0=\sum_{j=0}^s g_{s-j}(x) \HS(K_j,x).$$
Since $\HS(K_j,x)$ are powers series with non-negative terms and
$g_i(x)$ are polynomials with least degree term coefficient equal to
$1$, we conclude that $\HS(K_j,x)=0$ for $j=0,\dots,s$. Hence $K_j=0$
for $j=0,\dots,s$.  
\end{proof}

\begin{proof}[Proof of Theorem \ref{BSZB}] 
We may assume without loss of generality that $K$ is infinite. 
Let $A=(a_{ij})$ be an $m\times n$ matrix with entries in $K^*$, such
that all its $m$-minors are non-zero. It exists because  $K$ is
infinite.  Consider the $K$-algebra map $$\Phi:S=K[x_{ij} : 1\leq i
\leq m,\ 1\leq j \leq n]\to K[y_1,\dots,y_n]$$ induced by
$\Phi(x_{ij})=a_{ij}y_j$ for every  $i,j$. By construction the kernel
of $\Phi$  is generated by $n(m-1)$ linear forms. Let
$Y=\Phi(X)=(a_{ij}y_j)$. Denote by $[c_1,\dots,c_m]_W$ the minor with
column indices $c_1,\dots,c_m$ of an $m\times n$ matrix $W$. By
construction  
$$\Phi([c_1,\dots,c_m]_X)=[c_1,\dots,c_m]_Ay_{c_1}\cdots y_{c_m}.$$
Hence, by our assumption on $A$, we have that 
$$\Phi(I_m(X))=I_m(Y)=( y_{c_1}\cdots y_{c_m} : 1\leq c_1<\dots<c_m\leq n )$$ 
i.e., $I_m(Y)$ is generated by all the square-free monomials in
$y_1,\dots, y_n$ of total degree $m$. In particular it has codimension
$n-m+1$. It follows that  $I_m(Y)$ is resolved by the Eagon-Northcott
complex, hence $\Ker\Phi$ is generated by a regular sequence on
$S/I_m(X)$.  
Now let $\prec$ be any term order on $S$ and let $D$ be the ideal
generated by the leading terms of the maximal minors of $X$ with
respect to $\prec$. We have $D\subseteq \ini_\prec(I_m(X))$ and  
$$\Phi(\ini_\prec( [c_1,\dots, c_m]_X))=\Phi(x_{\sigma_1c_1}\cdots
x_{\sigma_mc_m})=a_{\sigma_1c_1}\cdots a_{\sigma_mc_m} y_{c_1}\cdots
y_{c_m}$$ for some $\sigma\in S_m$. Hence  
$$\Phi(D)=I_m(Y). $$ 
We apply Lemma~\ref{help} to the following data: 
$$M=S/I_m(X), \quad T=S/D,  \quad N=S/\ini_\prec(I_m(X)) \mbox{ and } J=\Ker \Phi$$
 to conclude that $D=\ini_\prec(I_m(X))$, and the Betti numbers of $I_m(X)$ equals those of $D$. 
\end{proof}

Can one generalize Theorem \ref{BSZB} to ideals of maximal minors of
matrices of linear forms? In Sections  \ref{colgrad} and
\ref{rowgrad}  we will give positive answers to the question by
assuming the matrix is  multigraded, either by rows or by columns. In
general however one cannot expect too much, as the following  remark
shows.  

\begin{remark} 
\label{ex23} 
One can consider various properties related to the existence of
Gr\"obner bases  and various families of matrices of linear forms.  
For instance we can look at the following properties for the ideal
$I_m(L)$ of $m$-minors of an $m\times n$ matrix $L$ of linear forms in
a polynomial ring $S$: 
\begin{itemize}
\item[(a)] $I_m(L)$ has a Gr\"obner basis of elements of degree $m$
  with respect to some  term order and possibly after a change  of
  coordinates.  
\item[(b)] $I_m(L)$ has a Gr\"obner basis of elements of degree $m$
  with respect to some  term order  and in the given coordinates.  
\item[(c)]  Property (b) holds and the associated initial ideal has a linear resolution. 
\item[(d)] $I_m(L)$ has a universal Gr\"obner basis of elements of degree $m$. 
\end{itemize}
We consider the following families of matrices of linear forms: 
\begin{itemize}
\item[(1)] No further assumption on $L$ is made.    
\item[(2)] $I_m(L)$ has codimension $n-m+1$.
\item[(3)] The entries of $L$ are linearly independent over the base
  field (i.e., $L$ arises from a matrix of variables by a change of
  coordinates). 
\end{itemize}

What we know (and do not know) is summarized in the following table: 

\begin{table}[htdp]
\begin{center}
\begin{tabular}{|c|c|c|c|c|}
\hline
     & (a) &  (b) & (c) & (d)   \\
\hline
(1) & no & no  & no & no\\ 
\hline
(2) & yes &no & no & no \\ 
\hline
(3)  & yes  & ? & ? & no  \\ 
\hline
\end{tabular}
\end{center}
\label{Who has what}
\end{table} 

There are ideals of $2$-minors of $2\times 4$ matrices of linear forms
that   define non-Koszul ring (see \cite[Remark 3.6]{C}). Hence those
ideals cannot have a single Gr\"obner bases of quadrics (not even
after a change of coordinates). This explains the four ``no" in the
first row of the table.  

Every initial ideal of the ideal of $2$-minors of 
$$\left(\begin{array}{ccc}
x_1+x_2 & x_3    & x_3   \\
   0   &   x_1& x_2
\end{array} 
\right)$$
has a generator in degree $3$ if the characteristic of the base field
is $\neq 2$. The codimension of $I_2(L)$ is $2$. This example explains
the three ``no" in the second row of the table.  
The ``yes" in the second row follows because the generic initial ideal
with respect to the reverse lexicographic order is generated in degree
$m$ under assumption (2).  

Finally, the matrix 
$$\left(\begin{array}{ccc}
x_1 & x_4         & x_3 \\
x_5 &x_1+x_6  & x_2
\end{array} 
\right)$$
belongs to the family (3) and the initial ideal with respect to any
term order satisfying $x_1>x_2>\dots>x_6$ has a generator in degree
$3$.   
This explains the ``no" in the third row. The ``yes" is there because (3) is contained in (2). 
 
It remains open whether the ideal of maximal minors of a matrix in the
family (3) has at least a Gr\"obner basis of elements of  degree $m$
in the given coordinates, and whether the associated initial ideal has
a linear resolution.  
\end{remark}

\section{Radical and Borel fixed ideals} \label{radBor}

The goal of the section is to prove Theorem \ref{onlyone},  a rigidity
result for multigraded Hilbert series associated to radical
multigraded Borel fixed ideals. Special cases of it appeared already
in \cite{CS} and \cite{ATS}. We will introduce the geometric
multidegree, a generalization of the notion of multidegree of Miller
and Sturmfels  \cite[Chap.8]{MS}, that allows us to deal with minimal
components of various codimensions in the case of Borel fixed ideals.  
 
Given $m\in \NN$ and $(n_1,\dots,n_m)\in \NN^m$  let $S$ be the
polynomial ring in the  set of variables $x_{ij}$ with $1\leq i\leq m$
and $1\leq j\leq n_i$ over an infinite field $K$, with grading induced
by $\deg(x_{ij})=e_i\in \ZZ^m$.   
Let $M$ be a finitely generated, $\ZZ^m$-graded $S$-module. 
The multigraded Hilbert series of $M$ is: 
 $$\HS(M,y)=\HS(M,y_1,\dots,y_m)=\sum_{a\in \ZZ^m} (\dim M_a)  y^a\in
 \QQ[|y_1,\dots,y_m|][y_1^{-1},\dots, y_m^{-1}|].$$ 
 The $\KK$-polynomial of $M$ is: 
$${\KK}(M,y)={\KK}(M,y_1,\dots,y_m)=\prod _{i=1}^m  (1-y_i)^{n_i}\HS(M,y).$$
Indeed, 
$${\KK}(M,y)\in \ZZ[y_1,\dots,y_m][y_1^{-1},\dots, y_m^{-1}].$$
 
The group $G=\GL_{n_1}(K)\times \cdots \times \GL_{n_m}(K)$ acts on
$S$ as the  group of $\ZZ^m$-graded $K$-algebras automorphisms.  Let
$B=B_{n_1}(K)\times \cdots  \times B_{n_m}(K)$ be the Borel subgroup
of $G$ consisting of the upper triangular matrices with arbitrary
non-zero diagonal entries. An ideal $I$ is said to be Borel fixed if
$g(I)=I$ for every $g\in B$. Borel fixed ideals are monomial ideals
that can be characterized in a combinatorial way by means of exchange
properties as it is explained in \cite[Thm. 15.23]{E}. Indeed in
\cite[Thm. 15.23]{E}  details are given in the standard graded setting
but, as observed in \cite[Sect.1]{ACD},  the same characterization
holds also in the multigraded setting. Given a term order $\prec$
such that $x_{ik}\prec x_{ij}$ for $j>k$, one can associate a
(multigraded) generic initial ideal $\gin_{\prec}(I)$ to any
$\ZZ^m$-graded ideal of $I$ of $S$. $\gin_{\prec}(I)$ is Borel fixed.

 The prime Borel fixed ideals are easy to describe. Set 
 $$U= \{ (b_1,\dots, b_m) \in \NN^m : b_i\leq  n_i \mbox{ for every } i=1,\dots,m\}.$$
The following assertion follows immediately from the definition. 

 \begin{lemma}\label{easy1}
For every vector $b\in U$  the ideal 
 $$P_b=( x_{ij}  : i=1,\dots,m \mbox{ and } 1\leq j\leq b_i )$$ 
 is prime and Borel fixed, and every prime  Borel fixed ideal is of this form. 
 \end{lemma} 
 
\begin{lemma}\label{easy2} 
The associated prime ideals of a Borel fixed ideal $I$ are Borel fixed. 
\end{lemma}
\begin{proof}  
Let $P$ be an associated prime to $S/I$.  Clearly $P$ is monomial
(i.e., generated by variables) because $I$ is monomial. We have to
prove that if $x_{ij}\in P$ then also $x_{ik}\in P$ for all $k<j$. We
may write $P=I:f$ for some monomial $f$. Let $\alpha$ be the exponent
of $x_{ij}$ in $f$. Consider $g\in B$ such that
$g(x_{ij})=x_{ij}+x_{ik}$ and fixes all the other variables. Then
$g(x_{ij}f)\in I$ because $x_{ij}f\in I$. The monomial
$x_{ik}^{\alpha+1}f/x_{ij}^{\alpha}$ appears with nonzero coefficient
in $g(x_{ij}f)$. Hence  $x_{ik}^{\alpha+1}f/x_{ij}^{\alpha}\in I$ and
$x_{ik}^{\alpha+1}f\in I$.  In other words, $x_{ik}^{\alpha+1}\in
I:f=P$ and hence  $x_{ik}\in P$.  
\end{proof} 
  
\begin{lemma}\label{easy3}  
Let $I$ be a radical and Borel fixed ideal.  Then every minimal generator of $I$ has multidegree bounded above by $ (1,1,\dots,1)\in \ZZ^m$.  
\end{lemma} 
\begin{proof} 
Consider a generator $f$ of $I$ of degree $(a_1,\dots,a_m)$.  We may
write $f=ug$ with $u$ a monomial of degree $a_1e_1$. Since $I$ is
Borel fixed we have $x_{1j}^{a_1}g\in I$, where $j=\min\{ k  :
x_{1k}|u\}$. Since $I$ is radical we have  $x_{1j}g\in I$, and
$x_{1j}g$ is a proper divisor of $f$ unless $a_1=1$.  
 \end{proof} 
 
\begin{lemma}\label{easy4}  
Let  $I$  be a radical Borel fixed ideal and let $\{P_{b_1}, \dots,
P_{b_c}\}$, with $b_1,\dots,b_c\in U$, be the  minimal primes of
$I$. Then $I$ is the Alexander dual of the polarization of  
$$J=(\prod_{b_{ij}>0}  x_j^{b_{ij}} : i=1,\dots, c)\subset K[x_1,\dots,x_m].$$
In particular, if all the generators of $I$ have the same multidegree,
then $I$ has a linear resolution.  
\end{lemma} 
\begin{proof} 
The first assertion follows immediately from the definition of
polarization and Alexander duality, see \cite[Chap.5]{MS}. For the
second, one observes that if all the generators of $I$ have degree,
say,  $e_1+e_2+\cdots+e_u \in \ZZ^m$, then $I$  is the Alexander dual
of the polarization of an ideal $J\subset K[x_1,\dots,x_m]$ involving only
variables $x_i$ with $i\le u$ and whose radical is
$(x_1,\dots,x_u)$. Hence $J$ defines a Cohen-Macaulay ring, and so
does its polarization. Finally one applies the Eagon-Reiner Theorem
\cite[Thm. 8.1.9]{HH}.  
\end{proof} 

The goal of this section is to prove the following:
 
\begin{theorem}\label{onlyone}
Let  $I,J\subset S$ be  Borel fixed ideals  such that
$\HS(I,y)=\HS(J,y)$.  If $I$ is radical then $I=J$.  
\end{theorem} 

The most important consequence of Theorem~\ref{onlyone} is the
following rigidity result:  

\begin{corollary}\label{multigin}
Let $I$ be a radical Borel fixed ideal. For every multigraded ideal
$J$ with $\HS(J,y)=\HS(I,y)$  one has:  
\begin{itemize}
\item[(a)]  $\gin_\prec(J)=I$ for every term order $\prec$. 
\item[(b)]  $J$ is radical. 
\item[(c)]  $J$ has a linear resolution whenever $I$ has a linear resolution. 
\item[(d)]  $S/J$ is Cohen-Macaulay whenever $S/I$  is Cohen-Macaulay. 
\item[(e)] $\beta_{i,a}(S/J)\leq \beta_{i,a}(S/I)$ for every $i\in
  \NN$ and $a\in \ZZ^m$ and $\beta_{i,a}(S/J)=0$ if  $a\not\leq
  (1,1,\dots,1)\in \ZZ^m$.  
\end{itemize}  
\end{corollary} 
\begin{proof} The ideal  $\gin_\prec(J)$ is a Borel fixed ideal and
  $\HS(J,y)=\HS(\gin_\prec(J),y)$. Since, by assumption,
  $\HS(J,y)=\HS(I,y)$, we may conclude, by virtue of
  Theorem~\ref{onlyone} that $\gin_\prec(J)=I$. This proves (a).
  Statements (b),(c) and (d)  are standard applications of  well-known
  principles.  Finally (e) follows from Lemma~\ref{easy3} and from the
  bounds derived form the Taylor complex, see \cite[Chap.6]{MS} 
\end{proof}

In order to prove Theorem~\ref{onlyone} we need the following definition. 

\begin{definition} For every finitely generated $\ZZ^m$-graded $S$-module $M$ we set 
$$\CC(M,y)=\KK(M,1-y_1,\dots, 1-y_m)\in  \ZZ[|y_1,\dots,y_m|]$$
and we define the G-multidegree (geometric multidegree) of $M$ as
$$\GG(M,y)=\sum  c_a y^a  \in  \ZZ[y_1,\dots,y_m]$$
where the sum is over the $a\in \ZZ^m$ which are minimal in the
support of ${\CC}(M,y)$ and $c_a$ is the coefficient of $y^a$ in
$\CC(M,y)$.  
\end{definition} 

The following result follows immediately from the definition above. 

\begin{proposition}\label{easyp}
\begin{itemize} 
 \item[(1)] Let  $P$ be a prime ideal generated by variables and  let
   $a(P)$ be the vector whose $i$-th  coordinate is $\# (P\cap
   \{x_{i1}, \dots, x_{in_i}\})$. Then  
$$\GG(S/P,y)=y^{a(P)}$$
\item[(2)] One has $a(P_{b})=b$ for every $b\in U$ and for $b_1,b_2\in
  U$ one has $P_{b_1}\subseteq  P_{b_2}$ if and only if $y^{b_1} |
  y^{b_2}$. 
\end{itemize} 
\end{proposition} 

The key observation is  the following: 

\begin{proposition}\label{menomale} 
 Let  $I$ be a Borel fixed ideal. One has 
$$\GG(S/I,y)=\sum_{i=1}^c \length( (S/I)_{P_{b_i}})  y^{b_i}$$
where  $\Min(I)=\{P_{b_1},\dots, P_{b_c}\}$ for some $b_i\in U$. 
\end{proposition} 

\begin{proof}  
In order to compute the $\KK$-polynomial of $M=S/I$, consider a
filtration of $\ZZ^m$-graded modules  
$$0=M_0\subset M_1\subset \dots \subset M_h=M$$
such that $M_i/M_{i-1}\simeq S/P_i(-v_i)$. Here $P_i$ is a
$\ZZ^m$-graded monomial prime ideal and $v_i=(v_{i1},\dots, v_{im})
\in \ZZ^m$.  
Existence of such a filtration follows from basic commutative algebra
facts, see \cite[Prop.3.7]{E}.
Furthermore 
$$\Min(I)\subseteq \Ass(S/I)\subseteq  \{P_1,\dots,P_h\}.$$ 
Hence we have 
$$\KK(S/I,y)=\sum_{i=1}^h \KK(S/P_i(-v_i),y)=\sum_{i=1}^h y^{v_i}  \KK(S/P_i,y).$$
It follows that 
$$\CC(S/I,y)=\sum_{i=1}^h \prod _{j=1}^m (1-y_j)^{v_{ij}}  \CC(S/P_i,y).$$
 
Then the support of the polynomial $\prod _{j=1}^m
(1-y_j)^{v_{ij}}\CC(S/P_i,y)$ contains exactly one minimal element,
namely $y^{a(P_i)}$, which appears in the polynomial with coefficient
$1$.  
 It follows that $\GG(S/I,y)$ is obtained as the sum of the terms
 which are minimal in the support of the polynomial 
\begin{equation}\label{supp}
\sum_{i=1}^h y^{a(P_i)} 
\end{equation}    
Now the elements that are minimal support in the support of
\eqref{supp} are exactly the $y^{b_i}$ corresponding to the minimal
primes $P_{b_i}$. This follows from Proposition~\ref{easyp}, since if
$P\subseteq  P'$, then $y^{a(P)} | y^{a(P')}$. Finally, by standard
localization arguments we have that each minimal prime $P_{b_i}$
appears in the multiset $\{P_1,\dots,P_h\}$ as many times as $\length(
(S/I)_{P_{b_i}})$.  
 \end{proof}
 
We are finally ready to prove Theorem \ref{onlyone}. 
 
\begin{proof}[Proof of Theorem \ref{onlyone}] 
Since $I$ and $J$ have the same Hilbert series we have that 
$\CC(S/I,y)=\CC(S/J,y)$ and hence 
$$\GG(S/I,y)=\GG(S/J,y).$$
It follows by Proposition~\ref{menomale} that $\Min(I)=\Min(J)$. Since
$I$ is radical, the coefficients in $\GG(S/I,y)$ are all $1$. Hence
the primary decomposition of $J$ is of the form $I\cap Q$, where $Q$
is the intersection of the components associated to the embedded prime
ideals of $J$, if any.  
We deduce that $J\subseteq I$ and the Hilbert series forces the equality $I=J$.
\end{proof}

\section{Column-graded ideals of maximal minors}\label{colgrad}
 
Consider $S=K[x_{ij} : 1\leq i \leq m,  \quad 1\leq j \leq n]$ graded
by $\deg(x_{ij})=e_j\in \ZZ^n$.  
Let $L=(L_{ij})$ be a $m\times n$ matrix of linear forms which is
column-graded, that is, whose entries $L_{ij}$ satisfy
$\deg(L_{ij})=e_j$.  
In other words, $$L_{ij}=\sum_{k=1}^m  \lambda_{ijk} x_{kj}$$ where
$\lambda_{ijk}\in K$.  
As a first direct application of Corollary~\ref{multigin} we have: 

\begin{theorem}
\label{superBSZ1} 
Let $L=(L_{ij})$ be a $m\times n$ matrix which is column-graded and
assume that the codimension of $I_m(L)$ is $n-m+1$. Then $I_m(L)$  is
radical and the maximal minors of $L$  form a universal Gr\"obner
basis of it. Furthermore every initial ideal of $I_m(L)$ is radical,
has a linear resolution, and its Betti numbers equals those of
$I_m(L)$.  
\end{theorem} 

\begin{proof} We may assume without loss of generality that $K$ is
  infinite. Let $I=( x_{1j_1}x_{1j_2}\cdots x_{1j_m} : 1\leq
  j_1<j_2<\dots< j_m \leq n)$. Then $I$  is generated by the maximal
  minors of a column-graded  matrix  whose $(i,j)$-th entry is
  $a_{ij}x_{1j}$ with randomly chosen scalars $a_{ij}$. Since the
  codimension of $I$ is $n-m+1$, by the Eagon-Northcott complex it
  follows that $I$ and $I_m(L)$ have the same multigraded Hilbert
  series. Since $I$ is radical and Borel fixed, we may apply Corollary
  \ref{multigin} with $J=I_m(L)$ or $J$ equal any initial ideal of
  $I_m(L)$ to conclude.  
\end{proof} 
 
We want now to generalize Theorem \ref{superBSZ1} and get rid of the
assumption on the codimension of $I_m(L)$.  
\begin{theorem}
\label{superBSZ2} 
Let $L=(L_{ij})$ be an $m\times n$ matrix which is column-graded.  Then:
\begin{itemize}
\item[(a)]  The maximal minors of $L$ form a universal Gr\"obner basis of $I_m(L)$.
\item[(b)] $I_m(L)$ is radical and it  has a linear resolution.
\item[(c)]  Any initial ideal $J$ of $I_m(L)$ is radical and has a
  linear resolution.   In particular,
  $\beta_{i,j}(I_m(L))=\beta_{i,j}(J)$ for all $i,j$.  
\item[(d)]  Assume that $I_m(L)\neq 0$ and that no column of $I_m(L)$
  is identically $0$.  Then the projective dimension of $I_m(L)$ (and
  hence of all its initial ideals) is $n-m$.  
\end{itemize} 
\end{theorem} 
 
\begin{proof} Again we may assume that $K$ is infinite. Fix a term
  order $\prec$. It is not restrictive to assume that $x_{1j}\succ
  x_{ij}$ for all $i\neq 1$ and $j$; set for simplicity $x_j=x_{1j}$.   
Let $$I=(x_{j_1}\cdots x_{j_m}\mid [j_1,\ldots,j_m]_L\neq 0).$$  
We claim that $I=\gin_{\prec}(I_m(L))$. First we note that
$I\subseteq\gin_{\prec}(I_m(L))$.  This is because if
$[j_1,\ldots,j_m]_L\neq 0$, then $I_m(L)$ contains a non-zero element
of degree $e_{j_1}+\cdots+e_{j_m}$ and its initial term in generic
coordinates is $x_{j_1}\cdots x_{j_m}$.  

Next note that $I$ is the Stanley-Reisner ideal of the Alexander dual
of the matroid dual $M_L^*$ of the matroid $M_L$ associated to $L$. As
such, $I$ has a linear resolution by the Eagon-Reiner Theorem
\cite[Thm.8.1.9]{HH}, since $M_L^*$ is Cohen-Macaulay. By Buchberger's
Algorithm, in order to prove that $I=\gin_{\prec}(I_m(L))$ it suffices
to show that any $S$-pair associated to a linear syzygy  
among the generators of $I$ reduces to $0$. Any such linear syzygy
involves at most $m+1$ column indices in total. After renaming the
column indices, we may assume that  the syzygy in question involves
the column indices $\{1,2,\dots, m+1\}$. Set  
$$d=e_1+e_2+\cdots+e_{m+1}.$$ To prove that the $S$-polynomial reduces
to $0$ we may as well prove that $\dim I_m(L)_d\leq \dim I_d$.  
Let 
$$W=\{ u : 1\leq u\leq m+1 \mbox{ and }  [\{1,\dots, m+1\}\setminus
\{u\}]_L\neq 0\}.$$ Renaming if needed, we may assume that  
$$W=\{1,2,\dots,s\}$$
By definition $I_d$ is generated by the set of monomials 
$$\left\{ \frac{x_1x_2\cdots x_{m+1}}{x_j} x_{ij} \; :\; j=1,\dots s
  \mbox{ and  } i=1,\dots m\right\}$$ 
whose cardinality is easily seen to be $sm-s+1$. 
Hence it remains to prove that  
$$ \dim I_m(L)_d \leq sm-s+1.$$
Denote by $\Omega$ the first syzygy module of  $\{ [\{1,\dots,
m+1\}\setminus \{u\}]_L  : u=1,\dots, s\}$. Since  
$$\dim I_m(L)_d=sm-\dim\Omega_d $$ it suffices to show that
$$\dim \Omega_d \geq s-1.$$
 
Let $L_1$ be the submatrix of $L$ consisting of the first $s$ columns of $L$. 
Since the rows of $L_1$ are elements of $\Omega_d $, it is enough to
show that  that $L_1$ has at least $s-1$ linearly independent rows
over $K$. By contradiction, if this is not the case, by applying
invertible $K$-linear operations to the rows of  $L$ we  may assume
that  the last $m-s+2$ rows of $L_1$ are identically zero.  In
particular the minor $[2,\ldots,m+1]_{L}=0$, contradicting our
assumptions. 

Since $I$ is Borel fixed and radical with $\HS(I,y)=\HS(I_m(L),y)$, we
may apply Corollary~\ref{multigin} and deduce (a), (b) and (c). For
(d) one observes that, under the assumption that no column of $L$ is
$0$ and $I_m(L)\neq 0$, the ideal $I$ is non-zero and each of the
variables $x_1,\dots,x_n$ is involved in some generator. Then $M_L^*$
has dimension $n-m$ and has no cone-points.  This implies that  the
Stanley-Reisner ring  of $M_L^*$ has regularity $n-m$, as it is
$2$-Cohen-Macaulay (see \cite[pg.94]{S} for details). By
\cite[Prop.8.1.10]{HH} the projective dimension of  $I$ (that is the
Alexander dual of $M_L^*$) is $n-m$. 
\end{proof}

\section{Row-graded ideals of maximal minors} \label{rowgrad} 

In this section we treat  ideals of maximal minors of row-graded
matrices.  Consider $S=K[x_{ij} : i=1,\dots, m \mbox{ and } j=1,\dots,
n]$ graded by $\deg(x_{ij})=e_i\in \ZZ^m$.  
Let $L=(L_{ij})$ be a $m\times n$ matrix of linear forms which is
row-graded, i.e., whose entries $L_{ij}$ satisfy
$\deg(L_{ij})=e_i$. In other words,  
 $$L_{ij}=\sum_{k=1}^m  \lambda_{ijk} x_{ik}$$
 where $\lambda_{ijk}\in K$.  
 Observe that in the row-graded case  we cannot expect that the
 maximal minors of $X$ form a Gr\"obner basis simply because every
 maximal minor might have the same leading term. Nevertheless we can
 prove the following: 

\begin{theorem}\label{superBSZ3}
 
Let $L=(L_{ij})$ be an $m\times n$ matrix which is row-graded and
assume that the codimension of $I_m(L)$ is $n-m+1$. Then $I_m(L)$  is
radical and every initial ideal is generated by elements of total
degree $m$ (equivalently, there is a universal Gr\"obner basis  of
elements of degree $m$).   Furthermore every initial ideal of $I_m(L)$
is radical,  has a linear resolution, and its Betti numbers equals
those of $I_m(L)$.  
\end{theorem} 

%Theorem \ref{superBSZ3}  follows immediately from
%Corollary~\ref{multigin} and from the following proposition: 
%\begin{proposition}
%Under the assumptions of Theorem \ref{superBSZ3}, the  $\ZZ^m$-graded
%Hilbert series of $I_m(L)$ equals that of $$I=( x_{1j_1}\cdots
%x_{mj_m} : j_1+\dots+j_m \leq n).$$ 
%\end{proposition} 
% by observing that $J$ is radical and Borel fixed. Notice that
% Corollary~\ref{multigin} also implies that $I=\gin_\prec(I_m(L))$
% for every term order $\prec$.  
 
Set $$I=( x_{1j_1}\cdots x_{mj_m} : j_1+\dots+j_m \leq n).$$  
Theorem \ref{superBSZ3}  follows immediately from
Corollary~\ref{multigin} and from the following proposition, by
observing that $I$ is radical and Borel fixed. Notice that
Corollary~\ref{multigin} also implies that $I=\gin_\prec(I_m(L))$ for
every term order $\prec$.  
 
\begin{proposition}
\label{TWZ}
Under the assumptions of Theorem \ref{superBSZ3}, the  $\ZZ^m$-graded
Hilbert series of $I_m(L)$ equals that of $I$. 
%$$J=( x_{1b_1}\cdots x_{mb_m} : b_1+\dots+b_m \leq n).$$
\end{proposition} 
 
\begin{proof}  The Hilbert series of $I_m(L)$ equals that of $I_m(X)$
  with $X=(x_{ij})$, because both ideals are resolved by the
  multigraded version of the  Eagon-Northcott complex. Hence we may
  assume without loss of generality that $L=X$.  We will show that
  $S/I_m(X)$ and $S/I$ have the same $\KK$-polynomial.  

Let  $\KK_{m,n}(y)$ be the $\KK$-polynomial of $S/I_m(X)$. By looking
at the diagonal initial ideal of $I_m(X)$ one obtains the recursion: 
$$\KK_{m,n}(y)=(1-y_m)\KK_{m,n-1}(y_1,\dots,y_m)+y_m\KK_{m-1,n-1}(y_1,\dots,y_{m-1}).$$
Solving the recursion or, alternatively,  by looking directly at the
multigraded version of the Eagon-Northcott complex,  one obtains:  
\begin{equation}
\label{rg0}
\KK_{m,n}(y)=1-(\prod_{i=1}^m y_i) \sum_{k=0}^{n-m} (-1)^k\binom{n}{m+k} h_k(y_1,\dots, y_m)
\end{equation} 
where $h_k(y_1,\dots,y_m)$ is the complete symmetric polynomial of
degree $k$,  i.e., the sum of all the monomials of degree $k$ in the
variables $y_1,\dots,y_m$.  

We now compute the $\KK$-polynomial of $S/I$. 
For $b\in [n]^m$ set  $x_b=x_{1b_1}x_{2b_2}\cdots x_{mb_m}$ so that 
$$I=( x_b : b\in \NN_{>0}^m \mbox{ and } |b|\leq n).$$ 
Extend the natural partial order,  i.e. $x_b\leq x_c$ if $b\leq c$ coefficientwise, 
 to a total order $<$ (no matter how). 
 For every $b\in [n]^m$ we have:
\begin{equation}
\label{rg1}
( x_c : x_c<x_b):x_b=(x_{ij} : i=1,\dots,m \mbox{ and } 1\leq j< b_i).
\end{equation} 

Filtering $I$ according to $<$ and using \eqref{rg1} one obtains: 
\begin{equation}
\label{rg2}
\KK(S/I,y)=1-y_1\dots y_m \sum_b \prod_{i=1}^m (1-y_i)^{b_i-1} 
\end{equation} 
where the sum $\sum_b$ is over all the $b\in \NN_{>0}^m$  and  $|b|\leq n$. 
Setting $c=b-(1,\dots,1)$ and replacing $b$ with $c$ in \eqref{rg2} we obtain:  
\begin{equation}
\label{rg3}
\KK(S/I,y)=1-y_1\dots y_m \sum_c \prod_{i=1}^m  (1-y_i)^{c_i}
\end{equation} 
where the sum $\sum_c$ is over all the $c\in \NN^m$  and  $|c|\leq n-m$. 
We may rewrite the last expression as: 
\begin{equation}
\label{rg4}
\KK(S/I,y)=1-y_1\dots y_m \sum_{k=0}^{n-m}  h_k(1-y_1,\dots, 1-y_m).
\end{equation} 

Taking into consideration \eqref{rg0} and \eqref{rg4}, it remains to prove that: 
\begin{equation}
\label{rg5}\sum_{k=0}^{n-m}  h_k(1-y_1,\dots, 1-y_m)= \sum_{k=0}^{n-m}
(-1)^k\binom{n}{m+k} h_k(y_1,\dots, y_m) 
\end{equation}  
 or equivalently, by replacing $y_i$ with $-y_i$ in \eqref{rg5}, it is left to show that: 
\begin{equation}
\label{rg6} 
\sum_{k=0}^{n-m}  h_k(1+y_1,\dots, 1+y_m)= \sum_{k=0}^{n-m} \binom{n}{m+k} h_k(y_1,\dots, y_m).
\end{equation}

 Setting $t=n-m$, \eqref{rg6}  is equivalent to the assertion that the equality:
\begin{equation}\label{rg7} 
\sum_{k=0}^{t}  h_k(1+y_1,\dots, 1+y_m)= \sum_{k=0}^{t} \binom{m+t}{m+k} h_k(y_1,\dots, y_m)
\end{equation}   
holds for every $m$ and $t$. 
The formula \eqref{rg7}  can be derived  from the more precise: 
\begin{equation}\label{rg8} 
h_t(1+y_1,\dots,1+y_m)=\sum_{k=0}^t \binom{m+t-1}{m+k-1} h_{k}(y_1,\dots,y_m).
\end{equation} 
 
Equation \eqref{rg8} can be proved by (long and tedious) induction  on
$m$.   The following simple argument using generating functions was
suggested by Christian Krattenthaler.  
First  notice that:
\begin{equation} 
\label{rg10}
\sum_{t\geq 0} h_t(y_1,\dots, y_m)z^t=\prod_{i=1}^m  \frac{1}{1-y_iz}\ .
\end{equation} 
Replacing  in \eqref{rg10} $y_i$ with $y_i+1$ and  observing that 
$$\prod_{i=1}^m  \frac{1}{1-(y_i+1)z}=\frac{1}{(1-z)^m} \prod_{i=1}^m
\frac{1}{1-y_i\frac{z}{(1-z)}}$$ 
we have:
\begin{equation} \label{rg11}\sum_{t\geq 0} h_t(1+y_1,\dots,
  1+y_m)z^t= \sum_{t\geq 0} h_t(y_1,\dots, y_m)
  \frac{z^t}{(1-z)^{t+m}}\ .\end{equation}  
Expanding the right-hand side of \eqref{rg11} one obtains  \eqref{rg8}. 
\end{proof}

 \end{document}